\documentclass[reqno,11pt]{amsart}
\usepackage[utf8]{inputenc}  
\usepackage{blkarray}
\usepackage{diagbox}
\usepackage{float} 
\usepackage[in]{fullpage}

\newcommand{\ds}{\displaystyle}

\newcommand{\tensor}{\otimes}
\newcommand{\leftsub}[2]{{\vphantom{#2}}_{#1}{#2}} 
\newcommand{\xycirc}[2]{\leftsub{#1}{\circ}_{#2}}

\newcommand{\op}{\mathcal}

\newcommand{\cdc}{,\dots,}

\newcommand{\FTGK}{\mathsf{FT}}

\usepackage{amsmath}%
\usepackage{amsthm}
\usepackage{amsfonts}%
\usepackage{amssymb}%
\usepackage{graphicx}
\usepackage{xy,amsthm,enumerate,xypic,array}  
\usepackage{xcolor}
\usepackage{eucal}

\usepackage{mathtools}

\DeclarePairedDelimiter\floor{\lfloor}{\rfloor}

\input{xy}
\xyoption{all}

\numberwithin{equation}{section}

\newtheorem{theorem}{Theorem}[section]
\theoremstyle{plain}

\newtheorem{conjecture}[theorem]{Conjecture}

\newtheorem{corollary}[theorem]{Corollary}
\newtheorem{lemma}[theorem]{Lemma}
\newtheorem{proposition}[theorem]{Proposition}

\theoremstyle{definition}
\newtheorem{definition}[theorem]{Definition}

\newtheorem{remark}[theorem]{Remark}

\allowdisplaybreaks[2]

\addtocounter{MaxMatrixCols}{2}

\begin{document}
	
\title{Stirling Decomposition of Graph Homology in Genus $1$.} 


\author{Benjamin C. Ward}
\email{benward@bgsu.edu}
\thanks{It is my sincere pleasure to thank Ralph Kaufmann, Martin Markl and Sasha Voronov for their invitation to participate in the Special Session on Higher Structures in Topology, Geometry and Physics at the AMS Spring Central Meeting, March 2022.  I would also like to thank an anonymous referee for comments which improved an earlier version of this article.}


\date{}

\begin{abstract}
We prove that commutative graph homology in genus $g=1$ with $n\geq 3$ markings has a direct sum decomposition whose summands have rank given by Stirling numbers of the first kind.  These summands are computed as the homology of complexes of certain decorated trees which have an elementary combinatorial description. 
\end{abstract}


\maketitle

\section{Introduction}

The purpose of this paper is to compute the homology of a family of chain complexes which arose in my study of higher operations on graph homology.  We will denote these chain complexes as $\op{S}_{n,k}$.  Although they arose from the study of some rather technical topics (eg $\infty$-modular operads and the Feynman transform), it is possible to give a very elementary description of them in terms of the combinatorics of decorated trees, and that is the approach we take here-in.  

The computation of the homology $H_\ast(\op{S}_{n,k})$ is a new result, given for the first time in this article.  Specifically, we prove (Theorem $\ref{main}$) that the Betti numbers of $\op{S}_{n,k}$ are
\begin{equation}\label{srank}
	\beta_i(\op{S}_{n,k}) =  \begin{cases} |s_{n,k}| & \text{ if } i=n \\ 0 & \text{ else,}    \end{cases}
\end{equation}
where $s_{n,k}$ denotes the signed Stirling number of the first kind, whose absolute value counts the number of permutations of the set $\{1\cdc n\}$ which can be written as a product of $k$ disjoint cycles.  

Let me briefly explain, in basic terms, the context in which this result can be applied.    By (commutative) graph homology, we refer to a chain complex spanned by isomorphism classes of graphs.  The degree of a graph is its number of edges, and the differential is given by a signed sum of edge contractions.  This complex splits over the genus $g$ of the graph.  A variant of this construction allows $n$ labeled markings of the vertices.  In this way, for each pair $(g,n)$ we get a chain complex of such graphs, and we endeavor to compute its rational homology.  The expert reader has surely noticed several important technical details which we've glossed over in this informal description.

This article concerns the calculation in the case $g=1$ and $n\geq 3$. In this case, the rational homology may be determined from \cite[Theorem 1.2]{CGP2}, where the authors determine the homotopy type of a cell complex $\Delta_{1,n}$ whose reduced cellular chain complex coincides (up to a shift in degree) with the graph complex described above.  In particular they show
\begin{equation}\label{rank}
	\beta_i(\Delta_{1,n}) =  \begin{cases}  \ds\frac{(n-1)!}{2} & \text{ if } i=n-1 \\ 0 & \text{ else,}    \end{cases}
\end{equation}
and moreover prove that as an $S_n$-module, $H_{n-1}(\Delta_{1,n})$ is the representation induced from a particular $1$-dimensional representation of the dihedral group.

Our homology calculation above (Equation $\ref{srank}$) gives a new proof of Equation $\ref{rank}$, and it gives a different description of the $S_n$-module structure which adds some computational facility.  Namely we prove (Corollary $\ref{gccor})$ that 
\begin{equation}\label{decomp}
H_{\ast}(\Delta_{1,n+1}) \cong \ds\bigoplus_{i=1}^{\floor{n/2}} H_\ast(\op{S}_{n,2i}),	
\end{equation}
and hence the rank of $H_{n-1}(\Delta_{1,n})$ is equal to the number of permutations in $S_{n-1}$ which can be written using an even number of disjoint cycles (counting fixed points as cycles of length 1).  Depending on the parity of $n$, these coincide with either the even or the odd permutations and in either case, it's exactly half.  Whence our confirmation of Equation $\ref{rank}$.

In addition to giving a new proof of Equation $\ref{rank}$, the decomposition in Equation $\ref{decomp}$ tells us something new about the $S_{n+1}$-module structure.  The isomorphism in Equation $\ref{decomp}$ holds $S_{n+1}$ equivariantly, meaning that the irreducible decomposition of a given $H_n(\Delta_{1,n+1})$ also splits accordingly.  Moreover, we conjecture (Section $\ref{ghsec}$) that each of the sequences $H_{n}(\op{S}_{n,n-i})$ is representation stable (after tensoring with the sign representation).  This means that the irreducible decomposition of a given $H_n(\Delta_{1,n+1})$ can be broken into pieces, all but one of which is determined by lower $n$.  It is not too difficult to compute the first few non-trivial stable sequences by hand, see Subsection $\ref{rs}$, although the representation stability aspects fall somewhat outside of our parameters here. 

Indeed the goal of this paper is to give a broadly accessible description of the complexes $\op{S}_{n,k}$ and to compute their homology.  We call these complexes ``Stirling complexes'' and they are defined in Section 3.  Section 2 gives a brief recollection of Stirling numbers and proves an identity (Proposition $\ref{tweedie}$) which is the shadow of our main theorem when interpreted in terms of the Euler characteristic.  The homology calculation is given in Section $\ref{homologysec}$, and we conclude with a discussion of the decomposition of Equation $\ref{decomp}$ in Section $\ref{ghsec}$.

Finally, it should be pointed out that in attempting to present this problem in a hands-on and elementary way, we have necessarily but perhaps unfairly circumvented the use of more sophisticated tools which could also be used to attack this problem.  This includes the theory of Grobner bases \cite{DK}, since our main calculation could be recast as the Koszulity of Lie graph homology in genus $1$, viewed as an infinitesmal bimodule over its underlying genus 0 operad.  The novice reader should also be advised that this article has largely bypassed a discussion of the influential and sophisticated literature on graph complexes for which the works \cite{Kont}, \cite{WTw}, \cite{KWM}, \cite{MGC} may serve as a possible starting point for further reading.

\subsection*{Conventions}  Symmetric groups are denoted by $S_n$, or $S_X$ for a finite set $X$.  The irreducible representation of the symmetric group associated to a partition $\lambda$ will be denoted $V_\lambda$.  We denote shift operators for graded vector spaces by $\Sigma^\pm$.  We work over the field of rational numbers throughout.  


\section{Stirling Numbers}

Let $S_n$ denote the symmetric group of permutations of the set $\{1\cdc n\}$.  Let $S_{n,k}\subset S_n$ denote the subset consisting of permutations having $k$ disjoint cycles, where fixed elements are considered cycles of length $1$.  For example, the magnitude of the sets $S_{n,k}$ for $n\leq 7$ is given in the following table.

\begin{equation}\label{stable}
\begin{tabular}{c|c|c|c|c|c|c|c}
	\diagbox{n}{k}  & 1 & 2 & 3 & 4 & 5 & 6 & 7 \\ \hline
	1 & 1  & & & & & &\\ \hline
	2 & 1  & 1 & & & & &\\ \hline
	3 & 2  & 3 & 1 & & & &\\ \hline
	4 & 6  & 11 & 6 & 1 & & &\\ \hline
	5 & 24  & 50 & 35& 10& 1 & &\\ 	\hline	
	6 & 120  & 274 & 225 & 85 & 15 & 1 &\\ \hline
	7 & 720  & 1764 & 1624& 735& 175 & 21 &1 \\ 	
\end{tabular}
\end{equation}
By convention we consider $S_{n,k}$ to be the empty set, so the blanks should be considered zeros.

This triangle of numbers has been long studied under the name of the (unsigned) Stirling numbers of the first kind.  We define the signed Stirling numbers of the first kind to be $$s_{n,k}:= (-1)^{n-k}|S_{n,k}|.$$  From now on we will just call these the ``Stirling numbers'' for short.

We refer to \cite[Chapter 8]{EC} for a comprehensive overview of the Stirling numbers and their history.  In this article we will restrict attention to the small bits and pieces of this theory which we need.  We first mention a few basic facts that we will use below.  Their verification is immediate.
\begin{lemma}\label{basics}  Stirling numbers of the first kind satisfy:
	\begin{enumerate}
		\item $\sum_{k} |s_{n,k}| = n!$
		\item $s_{n,n-1} = -\ds{n \choose 2}$
		\item $s_{n,n-2} = \ds\frac{1}{4}\ds{n \ds\choose 3}(3n-1)$
		\item If $n\geq 2$ then $\sum_{k} s_{n,k}=0$.
	\end{enumerate}
\end{lemma}

The main theorem of this paper will be a categorification of the following slightly less trivial fact about Stirling numbers.

\begin{proposition}\label{alt}
	\begin{equation}\label{tweedie}
s_{n,k}=\ds\sum_{m=k+1}^{n+1} { m-1 \choose k}s_{n+1,m}.
	\end{equation}
\end{proposition}
\begin{proof}
We emphasize that these are signed Stirling numbers, so we're considering an alternating sum and this result arises as a corollary of our main theorem (Theorem $\ref{main}$) by taking the Euler characteristic.  However, let us give the combinatorial proof which is the shadow of the proof of our main theorem below.
To unify the two, let us choose to regard $S_{n+1}$ as the permutation group of the set $\{0,1\cdc n\}$ via the isomorphism $\{0,1\cdc n\}\cong \{1\cdc n,n+1\}$ which fixes $1\cdc n$.

Fix $n$ and $k$ with $k \leq n$.  Given a permutation in $S_{n+1}$ with $m$ cycles, there are ${m-1 \choose k}$ ways to choose $k$ cycles which do not contain $0$.  We may therefore view each term in the right hand side of Equation $\ref{tweedie}$ as counting the number of permutations in $S_{n+1}$ with $m$ cycles, $k$ of which are distinguished, subject to the rule that $0$ is not in a distinguished cycle.  Let $X_{m}$ be the set of such choices.  In particular $|X_{m}| = {m-1 \choose k}|S_{n+1,m}|$.

The set $X_{m}$ can be written as a disjoint union indexed by the choice of a non-empty set $D\subset \{1\cdc n\}$ and a permutation $\sigma \in S_{D}$ corresponding to the product of the $k$ distinguished cycles in the given element.  Let $X_{m,D,\sigma}\subset X_m$ be the subset appearing in index $(D,\sigma)$.  Counting the number of elements in each index, we may write:
\begin{equation}\label{tweedie2}
\ds\sum_{m=k+1}^{n+1} { m-1 \choose k}s_{n+1,m} = \sum_{m=k+1}^{n+1} \ds\sum_{D \subset \{1\cdc n\}}\ds\sum_{\sigma \in S_{D,k} } (-1)^{m+n+1}|X_{m,D,\sigma}|
\end{equation}
where $S_{D,k}$ is the set of permutations of $D$ having $k$ cycles.
 
Note that $X_{m,D,\sigma}\subset X_m$ is nonempty only if $m-k\leq n+1-|D|$.  Therefore, reindexing the previous equation we have 
\begin{equation}\label{t3}
\ds\sum_{m=k+1}^{n+1} { m-1 \choose k}s_{n+1,m} =  \ds\sum_{D \subset \{1\cdc n\}}\ds\sum_{\sigma \in S_{D,k} } \sum_{m=k+1}^{n+1-|D|+k} (-1)^{m+n+1}|X_{m,D,\sigma}|
\end{equation}
By considering the product of the non-distinguished cycles, each element in $X_{m,D,\sigma}$ specifies a unique permutation of $\{0\}\cup(\{1\cdc n\}\setminus D)$ having $m-k$ cycles, and conversely, so $|X_{m,D,\sigma}| = (-1)^{m-k+n+1-|D|}s_{n+1-|D|,m-k}$.  Therefore, for each fixed $D$ and $\sigma$ we have:
\begin{align*}
 \sum_{m=k+1}^{n+1-|D|+k} (-1)^{m+n+1}|X_{m,D,\sigma}| = (-1)^{k+|D|} \sum_{m=k+1}^{n+1-|D|+k} s_{n+1-|D|,m-k} \\ = (-1)^{k+|D|} \sum_{i=1}^{n+1-|D|} s_{n+1-|D|,i}. \  \ \ \ \ \ \ 
\end{align*}

Applying the final statement of Lemma $\ref{basics}$, we see that this sum is zero unless $n+1-|D|=1$.    This happens if and only if $D=\{1\cdc n\}$, i.e.\ every non-zero element appears in a distinguished cycle.  Thus Equation $\ref{t3}$ becomes
\begin{equation*}
\ds\sum_{m=k+1}^{n+1} { m-1 \choose k}s_{n+1,m} = (-1)^{n+k} \ds\sum_{\sigma \in S_{n,k} }|X_{k+1,\{1\cdc n\},\sigma}|.
\end{equation*}
Finally, the set $|X_{k+1,\{1\cdc n\},\sigma}|$ counts the number permutations of $S_{n+1}$ which fix $0$ and have $k+1$ cycles, such that the product of the $k$ cycles not containing $0$ is equal to $\sigma$.  Clearly there is exactly one such permutation, namely $(0)\cdot \sigma$, from which we conclude
\begin{equation*}
\ds\sum_{m=k+1}^{n+1} { m-1 \choose k}s_{n+1,m} =  (-1)^{n+k}|S_{n,k}| = s_{n,k}.
\end{equation*}
\end{proof}

\section{Stirling Complexes}

In this section we will define a chain complex $\op{S}_{n,k}$ whose Euler characteristic may be calculated via the alternating sum in Proposition $\ref{alt}$ to satisfy $\chi(\op{S}_{n,k})= s_{n,k}$.  For lack of better terminology we call $\op{S}_{n,k}$ ``Stirling complexes''.

In Section $\ref{homologysec}$ below we will calculate the homology of the Stirling complexes, and in Section $\ref{ghsec}$ we will give an application of this computation to graph homology.  The purpose of this section is simply to give their definition.  The Stirling complexes arose in \cite{WardMP} as certain subcomplexes of the Feynman transform of Lie graph homology.  However, it is possible to give a much more down to earth definition of them, and that is the approach we take here.

Finally, we remark that the reader should be sure to not confuse the chain complex $\op{S}_{n,k}$, defined in this section, with the set $S_{n,k}$ defined in the previous section.  These objects live in different categories, and while both have symmetric group actions, the two actions are not comparable.  One relationship between these two objects will be given later on, in Theorem $\ref{main}$.

\subsection{Decorated Trees}
We will define the Stirling complexes as a span of certain decorated trees.  For a formal definition of a tree we refer to the Appendix, namely subsection $\ref{trees}$.  

Informally, a tree $\mathsf{t}$ is a 1-dimensional CW complex with is connected and simply connected.  The $0$-cells are called vertices, the set of which is denoted $V(\mathsf{t})$.  The $1$-cells are called edges, the set of which is denoted $E(\mathsf{t})$.  An $n$-tree is a tree along with a function $\ell\colon \{0\cdc n\}\to V(\mathsf{t})$ called the leg labeling.  The leg labeling is depicted graphically by $|\ell^{-1}(v)|$ line segments extending from $v$, labeled bijectively by the elements of the set $\ell^{-1}(v)$.  These line segments are called the legs of the tree.  From this perspective what we call ``edges'' might also be called ``internal edges'' in that they run between two vertices where-as legs do not.  

Our formal definition of a graph (Appendix $\ref{graphs}$) views edges as consisting of two halves, called flags, and views legs as consisting of a single flag.  With this convention, every vertex has a set of adjacent flags, which is canonically identified with the union of the sets of adjacent edges and adjacent legs.  The valence of a vertex is defined to be the order of this set.  An $n$-tree may be viewed as a directed graph, with edges directed toward the leg labeled by $0$.  This in turn specifies a unique output flag at each vertex, the remainder of the flags are called inputs.  A tree is called stable if it has at least 2 input flags.

For the remainder of this article, including in the following definition, all trees are assumed to be stable $n$-trees for some $n$.  The Stirling complexes will be spans of such trees along with additional decorations.

\begin{definition}  A Stirling tree $(\mathsf{t}, v,A)$ is a tree $\mathsf{t}$ along with the choice of the following additional data:
	\begin{enumerate}
		\item a vertex $v$ of $\mathsf{t}$, which we call the distinguished vertex (or DV for short).
		\item a distinguished subset $A$ of the input flags adjacent to the distinguished vertex, whose elements we call alternating flags, such that $|A|\geq 2$.
	\end{enumerate}
\end{definition}
A Stirling tree having $n+1$ legs and $k$ alternating flags is said to be of type $(n,k)$.  See Figure $\ref{fig:st}$.  We will often refer to Stirling trees simply as $\mathsf{t}$, leaving the decorations $v$ and $A$ implicit.

\begin{figure}
	\centering
	\includegraphics[scale=.6]{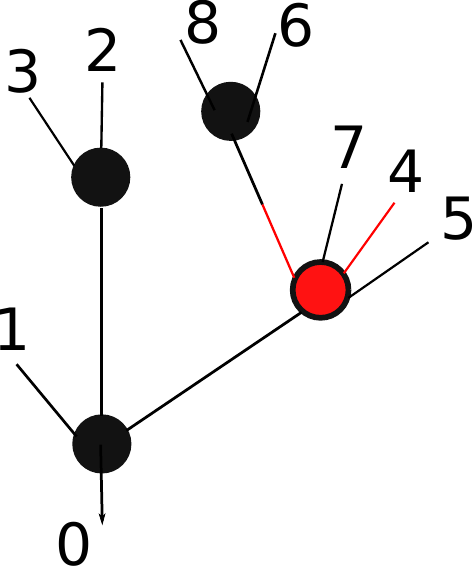}
	\caption{A Stirling tree with four vertices, three edges and nine legs.  The distinguished vertex and the alternating flags are depicted in red.  This Stirling tree is said to be of type $(8,2)$. 
	}
	\label{fig:st}
\end{figure}

We remark that we could loosen the requirement that $|A|\geq 2$ if we allowed trees whose distinguished vertex was not necessarily stable, but our intended application only requires this case, so we make this assumption for simplicity.

\subsection{The underlying spaces}\label{detsec}
In this article we consider Stirling trees up to structure preserving isomorphisms.  In particular, isomorphisms must preserve the leg labels, the distinguished vertex and the set of alternating flags.  Stirling complexes will be constructed as the sum over isomorphism classes of Stirling trees.

Each summand will in turn be spanned by a choice of additional sign data, which will be encoded as an order on the set of edges and the set of alternating flags, modulo even permutations.  To formalize this, given a finite set $X=\{x_1\cdc x_n\}$, we define $det(X)$ to be the top exterior power of the vector space spanned by $X$.  Explicitly, $det(X)$ is a $1$-dimensional vector space spanned by $x_1\wedge \dots \wedge x_n$ and subject to the rule that $x_1\wedge \dots \wedge x_n = sgn(\sigma) x_{\sigma(1)}\wedge \dots \wedge x_{\sigma(n)}$ for a permutation $\sigma \in S_X$.

We are now prepared to define the vector spaces underlying the Stirling complexes.
\begin{definition}  Define $\op{S}_{n,k,i}$ to be
	\begin{equation*}
	\op{S}_{n,k,i}=\bigoplus_{(\mathsf{t},v,A)} det(E(\mathsf{t}))\tensor det(A)
	\end{equation*}
with direct sum taken over all isomorphism classes of Stirling trees of type $(n,k)$ having $i$ edges.
\end{definition}

For example, the Stirling tree pictured in Figure $\ref{fig:st}$ specifies a one dimensional subspace of $\op{S}_{8,2,3}$.  The choice of an order on the set of edges and the set of alternating flags would further specify a non-zero element in this subspace.

\begin{lemma}\label{bounds}
	$\op{S}_{n,k,i}$ is non-zero if and only if $0\leq i \leq n-k$.
\end{lemma}
\begin{proof}  Suppose $\op{S}_{n,k,i}\neq 0$.  Then clearly $i\geq 0$.  Such a tree has $2i+n+1$ flags, but by stability must also have at least $3(|V|-1) + k+1$ flags, where $|V|$ denotes the number of vertices of the tree.  Using the fact that $|V|-1=i$ we see $n\geq i+k$ as desired.  Conversely, if we construct any tree with $n-k$ edges (for example by arranging all the vertices and edges along a vertical line), we may contract the edges one at a time to show that each such $\op{S}_{n,k,i}$ is not zero.
\end{proof}

\subsection{The symmetric group action}

The vector space $\op{S}_{n,k,i}$ is naturally an $S_n$-module via the permutation action on leaf labels.  For example, $\op{S}_{n,n,0}$ is the alternating representation and $\op{S}_{n,n-1,0}$ is the tensor product of the alternating representation with the permutation representation, i.e.\ $V_{1^n}\oplus V_{2,1^{n-2}}$.  However, the description of $\op{S}_{n,k,i}$ given in terms of the Feynman transform (Theorem $\ref{isothm}$ below) makes it clear that this ``obvious'' $S_n$ action is the restriction of an $S_{n+1}$-module structure.  Although not needed for the homology calculation below, let us describe this action in this subsection.

    First, we let $S_{n+1}$ act on the set of isomorphism classes of $n$-tress by permutation of the leg labels, using the unique isomorphism $\{0\cdc n\}\cong \{1\cdc n+1\}$ which fixes $1\cdc n$ and exchanges $0$ and $n+1$.  Let $\sigma \in S_{n+1}$ and $\mathsf{t}$ be an $n$-tree.  When comparing $\mathsf{t}$ and $\sigma\mathsf{t}$,
observe that there is a natural bijective correspondence between their edges, their vertices, and the flags at each vertex, although which flags are inputs and outputs need not be preserved.  Given a vertex $v$ of $\mathsf{t}$, we use the notation $\sigma v$ to denoted the corresponding vertex of $\sigma \mathsf{t}$.  Similarly for edges and flags, as well as sets of flags.

Now suppose we have a Stirling tree $(\mathsf{t}, v, A)$ of type $(n,k)$ and $\sigma \in S_{n+1}$.  The tree $\sigma\mathsf{t}$ comes with a distinguished vertex $\sigma v$, as well as a distinguished subset of the flags adjacent to $\sigma v$, namely $\sigma A$.  Since the edges and alternating flags are naturally identified, there is a map
\begin{equation}\label{sigma}
det(E(\mathsf{t}))\tensor det(A) \to det(E(\sigma\mathsf{t}))\tensor det(\sigma A)
\end{equation}
which simply preserves the order of each wedge product.  Note however since $\sigma$ needn't fix $n+1$, there is no guarantee that each flag in $\sigma A$ is an input flag.

With this in mind, we define the action of $S_{n+1}$ on $\op{S}_{n,k,i}$ in two cases.  Fix $\sigma \in S_{n+1}$.  Given a Stirling tree $(\mathsf{t},v,A)$, if $\sigma A$ does not contain the output flag of $\sigma v$, then we define the action of $\sigma$ on the summand of $\op{S}_{n,k,i}$ indexed by $(\mathsf{t},v,A)$ to be given by the map in Equation $\ref{sigma}$.

If, on the other hand, the output flag of $\sigma v$, call it $z$, is contained in the set $\sigma A$, we proceed as follows.  Let $B$ be the set of flags adjacent to $\sigma v$ which are not contained in $\sigma A$.  Note that $B$ is not empty, since under these conditions the output of $v$ corresponds to an input of $\sigma v$ which is not in $\sigma A$, hence is in $B$. 

For each $b\in B$, let $(\sigma A)_b := (\sigma A \setminus z) \cup\{b\}$.  The permutation $\sigma$ induces a natural bijective correspondence between the elements of $A$ and the elements of $(\sigma A)_b$ given by the intermediary $\sigma A$, and exchanging $z$ with $b$.  Therefore, there is a canonical map
\begin{equation*}
\sigma_b\colon det(E(\mathsf{t}))\tensor det(A) \to det(E(\sigma\mathsf{t}))\tensor det((\sigma A)_b)
\end{equation*}
given by canonically identifying the edges (resp.\ flags) on the left bijectively with those on the right.  In this case we define the action of $\sigma$ on the summand of $\op{S}_{n,k,i}$ indexed by $(\mathsf{t},v,A)$ to be given by the map
\begin{equation}\label{action}
\sum_{b\in B} -\sigma_b.
\end{equation}

We remark that in light of Theorem $\ref{isothm}$ below, this rule is merely a translation of \cite[Lemma 3.2]{WardWheels}, which in this terminology tells us how to rewrite a term containing a distinguished output flag as a linear combination of those having only distinguished input flags.

\subsection{The differential}\label{diffsec}

We are now prepared to define the chain complex $\op{S}_{n,k}$.  It will have the following form,
$$
0\to \op{S}_{n,k,n-k} \to \dots \to\op{S}_{n,k,i+1} \to \op{S}_{n,k,i} \to \dots \to \op{S}_{n,k,0} \to 0.
$$
By convention $\op{S}_{n,k,i}$ is considered to be of total degree $i+k$, so that $\op{S}_{n,k}$ is concentrated between degrees $k$ and $n$.  

The differential $d$ will be defined as a sum over edge contractions, as in the case of the bar construction of an operad \cite{GK}.  Specifically, this means $d$ applied to the summand of $\op{S}_{n,k,i}$ corresponding to a Stirling tree $(\mathsf{t},v,A)$ will be of the form  $d=\sum_{e\in \mathsf{t}} d_e$, and we proceed to define $d_e$.

Fix such a Stirling tree $(\mathsf{t},v,A)$ and choose an edge $e$ of $\mathsf{t}$.  We may contract the edge to produce a new tree $\mathsf{t}/e$.  If this edge does not contain an alternating flag, then the edge collapse specifies a Stirling tree whose distinguished vertex and alternating flags are canonically identified with those of the source.  Thus we have a map
\begin{equation}\label{edge}
det(E(\mathsf{t}))\tensor det(A_\mathsf{t}) \stackrel{d_e}\longrightarrow det(E(\mathsf{t}/e))\tensor det(A_\mathsf{t/e})
\end{equation}
which by convention removes the edge $e$ in the last position of the wedge product of the edges of $\mathsf{t}$.

Now suppose, on the other hand, that the edge $e=\{a,b_0\}$ consists of input flag $a\in A$ and output flag $b_0$.  When $e$ is contracted, the alternating flag $a$ is lost and we account for this by summing over the newly adjacent flags as replacements, see Figure $\ref{fig:std}$.  To make this precise, define $B$ to be the set of input flags at the vertex whose output is $b_0$, and define $A_{b}:= (A\setminus a) \cup \{b\}$, for each $b\in B$.  Observe that the elements of $A$ and $A_b$ are in canonical bijective correspondence, with bijection exchanging $a$ and $b$ and preserving all other elements.  This in turn induces an isomorphism $det(A)\to det(A_b)$.  We then define 
\begin{equation}\label{edgeb}
d_{e,b}\colon det(E(\mathsf{t}))\tensor det(A) \to det(E(\mathsf{t}/e))\tensor det(A_b)
\end{equation}
by taking this isomorphism on the right hand tensor factor and, as above, taking the map which removes $e$ in the last position of the wedge product on the left hand tensor factor.  Finally, we define $d_e := \sum_{b\in B} d_{e,b}$ in this case.

\begin{figure}
	\centering
	\includegraphics[scale=.85]{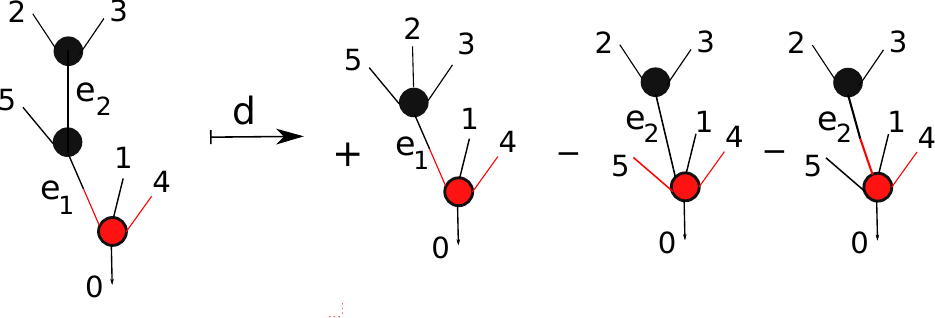}
	\caption{  The Stirling tree $\mathsf{t}$ on the left specifies an element of $\op{S}_{5,2,2}$, taking $e_1\wedge e_2 \in det(E(\mathsf{t}))$ and taking the alternating flags in the order indicated by the choice of planar embedding.  The differential applied to this element is indicated on the right.   The first differential term comes from contracting the edge which does not contain an alternating flag, the latter two come from contracting the edge which does contain an alternating flag.  One may verify that $d^2=0$ in this case by contracting the remaining edge in each of the three terms on the right.  The result is three pairs of identical terms appearing with opposite sign. }
	\label{fig:std}
\end{figure}

\begin{lemma} $d^2=0$.
\end{lemma}
\begin{proof}

	Start with a Stirling tree $(\mathsf{t},v,A)$.  We will show that $d^2=0$ when applied to the summand of $\op{S}_{n,k,i}$ indexed by this Stirling tree, from which the claim follows.   Note that if $\mathsf{t}$ has fewer than two edges, then $d^2=0$ just by degree considerations.  Else, choose two edges $e_1,e_2$ of $\mathsf{t}$, and we will argue that
$$d_{e_1}\circ d_{e_2}= -d_{e_2}\circ d_{e_1}$$
from which the statement follows.  We divide this verification into four cases.

{\bf Case I: } Suppose that neither $e_1$ nor $e_2$ contains an alternating flag.  In this case, we apply Equation $\ref{edge}$ twice and 
$d_{e_1}\circ d_{e_2}$ and $d_{e_2}\circ d_{e_1}$ are related by applying the transposition swapping $e_1$ and $e_2$, producing a factor of $-1$ on the $det(E(\mathsf{t}))$ tensor factor, while the alternating flags are preserved at each stage.

For the remaining cases, for $i\in \{1,2\}$, define $v_i$ to be the vertex adjacent and above $e_i$ and define $B_i$ to be the set of input flags adjacent to $v_i$.  

{\bf Case II: } Suppose that both $e_1$ and $e_2$ contain an alternating flag.  In this case, twice applying Equation $\ref{edgeb}$ we have
$$d_{e_1}\circ d_{e_2} = \ds\sum_{(b_1,b_2)\in B_1\times B_2} d_{e_1,b_1}\circ d_{e_2,b_2}$$
and similarly in the opposite order.  Here it suffices to observe that for each pair $(b_1,b_2)$ we have $d_{e_1,b_1}\circ d_{e_2,b_2} = -d_{e_2,b_2}\circ d_{e_1,b_1}$.  This follow as in the previous case, with the $-1$ arising from the $det(E(\mathsf{t}))$ tensor factor from the permutation of the two edges.  Since the wedge product of the alternating flags does not depend on the order of the edge contractions, no sign arises on the $det(A)$ factor. 

{\bf Case III: } Suppose (without loss of generality) that $e_1$ contains an alternating flag and $e_2$ does not, and suppose that $e_2$ is not adjacent to $v_1$.  In this case, applying Equation $\ref{edgeb}$ for $d_{e_1}$ and applying Equation $\ref{edge}$ for $d_{e_2}$ we have:
$$d_{e_1}\circ d_{e_2} = \ds\sum_{b\in B_1} d_{e_1,b}\circ d_{e_2}$$
and similarly in the opposite order.  It then suffices to verify $d_{e_1,b_1}\circ d_{e_2} = -d_{e_2}\circ d_{e_1,b_1}$, with the sign coming from the permutation of the two edges, as in the previous cases.

{\bf Case IV: }  We are thus left with the most interesting case, namely $e_1$ contains an alternating flag and $e_2$ is adjacent to $v_1$.  This case is depicted in Figure $\ref{fig:std}$.  Let $b_0 \in B_1$ be the input flag of the edge $e_2$.  If we contract $e_2$ first, the flags adjacent to $v_1$ are altered since $b_0$ is removed, but the flags in $B_2$ become adjacent to $v_1$.  Hence:
\begin{equation}\label{IV1}
d_{e_1}\circ d_{e_2} = \ds\sum_{b\in (B_1\setminus b_0)\cup B_2} d_{e_1,b}\circ d_{e_2}
\end{equation}
On the other hand, if we contract $e_1$ first we may single out the term $d_{e_1,b_0}$.  Since application of this map results in $b_0$ being an alternating flag, we must apply Equation $\ref{edgeb}$ when composing with $d_{e_2}$.  For each other $b_1\in B_1$, we apply Equation $\ref{edge}$.  Hence:
\begin{equation}\label{IV2}
d_{e_2}\circ d_{e_1} = \ds\sum_{b_1\in (B_1\setminus b_0)} d_{e_2}\circ d_{e_1,b_1} + \ds\sum_{b_2\in B_2} d_{e_2,b_2}\circ d_{e_1,b_0}
\end{equation}
To conclude, it suffices to observe that the terms in Equation $\ref{IV1}$ and $\ref{IV2}$ are negatives of each other.  This follows from the fact that for $b\in (B_1\setminus b_0) \cup B_2$, we have:
$$
-d_{e_1,b}\circ d_{e_2} = \begin{cases}
d_{e_2}\circ d_{e_1,b} & \text{ if } b\in B_1\setminus b_0 \\ 
d_{e_2,b}\circ d_{e_1,b_0} & \text{ if } b\in B_2
\end{cases}
$$
with the $-1$ factor again arising on the $det(E(\mathsf{t}))$ tensor factor by transposition of the edges.
\end{proof}

 The differential on each $\op{S}_{n,k}$ is compatible with the $S_{n+1}$-action defined in the previous subsection.  This follow from the identification of $\op{S}_{n,k}$ with a subcomplex of the Feynman transform, see Theorem $\ref{isothm}$, but it can also be verified directly from the above descriptions as we now indicate.
 
 \begin{proposition}  For each $\sigma\in S_{n+1}$, we have $\sigma d=d\sigma$.
 \end{proposition}
\begin{proof}

		Start with a Stirling tree $(\mathsf{t},v,A)$.  We will show that $\sigma d=d\sigma$ when applied to the summand of $\op{S}_{n,k,i}$ indexed by this Stirling tree, from which the claim follows.  If $\mathsf{t}$ has no edges the claim is vacuous, so assume $\mathsf{t}$ has at least one edge.  Since the edges of $\mathsf{t}$ and $\sigma\mathsf{t}$ coincide (i.e.\ are in canonical bijective correspondence) it suffices to show that $\sigma d_e = d_e\sigma$ for each edge $e$ of $\mathsf{t}$.
		
		Let us first observe that if $\sigma \in S_n$, i.e.\ if $\sigma$ fixes $0$, then the claim is immediate, since the permutation simply relabels the input legs of the tree.  Which flags are alternating does not change upon such a permutation, only their labels change, and this operation clearly commutes with edge contraction.  So it suffices to restrict attention to the case that $\sigma = (0 \ i)$ for some $1\leq i \leq n$ from which the general case follows.
		
		Fix $\sigma = (0 \ i)$ for such an $i$ and fix an edge $e$.  Let $a$ be the input flag adjacent to $v$ which is on the unique shortest path from the leg labeled by $i$ to $v$.  Let $e=\{x,y\}$ where $x$ is an input flag and $y$ is an output flag.  To show $\sigma d_e = d_e\sigma$, one separates the verification into several cases based on the possibilities for $a,x$ and $y$.  These cases are as follows:
		
\begin{enumerate}
	\item $a=x$ and $a\in A$,
	\item $a\in A$ and $y$ is adjacent to $v$,
	\item $a\in A$, $x$ is adjacent to $v$ and $x\not\in A$,
	\item $a,x\in A$, and $x\neq a$,
	\item none of the above, i.e. $a\not\in A$ or $e$ is not adjacent to $v$.
\end{enumerate}

Case 5 is the simplest, since the terms coincide on the nose with no cancellation.  The first four cases are very similar to each other, so we carefully work through Case 1 and leave the verification of the others to the reader.  %

For this, assume that $a=x$ and $a\in A$.  In this case the edge $e$ is adjacent to $v$.  Let $w$ be the vertex whose output is $y$ and let $B$ be the set of inputs of $w$.  The unique shortest path connecting $v$ to the leg labeled by $i$ necessarily contains one input of $w$, call it $b_0\in B$.  Finally, define $C$ to be the set of flags adjacent to $v$.  

By Equation $\ref{action}$ we have:
\begin{equation}\label{oneway}
d_e \sigma = d_e\circ(-\sum_{C\setminus a} \sigma_c) = - \sum_{C\setminus a}d_e\circ  \sigma_c
\end{equation}
On the other hand,
$$\sigma d_e=\sum_{b\in B} \sigma d_{e,b} = \sigma d_{e,b_0} + \sum_{b\neq b_0} \sigma d_{e,b} $$
Since the image of $\sigma d_{e,b_0}$ has a distinguished output, we apply Equation $\ref{action}$ to it to rewrite this term as $\sigma d_{e,{b_0}} = \sum -\sigma_c d_{e,{b_0}}$,
where here $c$ ranges over the input flags of the vertex of $\sigma \mathsf{t}$ formed by contracting $e$.  This set is (canonically identified with) $(B\setminus b_0) \cup (C \setminus a)$.  Therefore
$$\sigma d_e = \left(\sum_{c\in B\setminus b_0} -\sigma_c d_{e,{b_0}} + \sum_{c\in C\setminus a} -\sigma_c d_{e,{b_0}} \right) + \sum_{b\in B\setminus b_0} \sigma d_{e,b} $$
The first and third sums cancel.  Finally, observe that for $c\in C\setminus a$ we have $\sigma_c d_{e,b_0} = d_e\sigma_c$, since both simply contract the edge $e$ and replace $a$ with $c$ as the new alternating flag.  Hence the middle sum in the previous equation is exactly $d_e\sigma$ by Equation $\ref{oneway}$, as desired.
\end{proof}

\section{Homology of Stirling Complexes}\label{homologysec}


\begin{theorem}\label{main}  The Betti numbers of the complex $\op{S}_{n,k}$ are: $$
	\beta_i(\op{S}_{n,k}) = \begin{cases} |s_{n,k}| & \text{ if } i=n \\ 0 & \text{ else}  \end{cases}. $$
\end{theorem}
\begin{proof}  Fix $n,k$.  If $n=k$ the statement is trivial, so assume $n>k$.
	
	Define $\op{A}$	to be the span of those trees such that one or both of the following properties hold:
	\begin{enumerate}
		\item[(i)] the distinguished vertex has valence $> k+1$ (so not all inputs are alternating)
		\item[(ii)] the distinguished vertex is not the root vertex.
	\end{enumerate}
	Observe that $\op{A}\subset \op{S}_{n,k}$ is a subcomplex.  Indeed if a tree satisfies property (i) then differential terms arising from it will also satisfy property (i), while if a tree satisfies property (ii) then differential terms arising from it will satisfy property (i) when edges adjacent to the distinguished vertex are contracted, and will satisfy property (ii) when edges not adjacent to the distinguished vertex are contracted.  The heart of the proof is the following claim.
	
	{\bf Claim:} $\op{A}$ has no homology.
	
	To prove this claim, we will filter the complex $\op{A}$.  For a Stirling tree $\mathsf{t}\in \op{A}$ there is a unique shortest path connecting the distinguished vertex and the root vertex.  We'll refer to this path as the ``0-to-DV path''.  Let $P(\mathsf{t})\subset E(\mathsf{t})$ be the set of edges on this path.  Let $p(\mathsf{t})$ and $e(\mathsf{t})$ be the number of (path) edges, taking values $0\leq p(\mathsf{t}) \leq e(\mathsf{t})$.
	
	Define the reach of a tree to be $r(\mathsf{t}) = 2e(\mathsf{t})-p(\mathsf{t})- \nu(\mathsf{t})$, where $\nu(\mathsf{t})$ is the ``vcd check'' defined by
	\begin{equation*}
	\nu(\mathsf{t}):= \begin{cases} 1 & \text{ if the distinguished vertex has valence equal to $k+1$,} \\
							0 & \text{ if the distinguished vertex has valence not equal to $k+1$.} 
	\end{cases}
	\end{equation*}
We adopt the terminology ``vcd vertex'' to refer to the distinguished vertex of a Stirling tree $\mathsf{t}$ for which $\nu(\mathsf{t})=1$.
	
Since we've assumed $n>k$, the reach takes values in $0 \leq  r(\mathsf{t}) \leq 2(n-k)-2$.  It achieves the upper bound if $e(\mathsf{t})= n-k-1$, $\nu(\mathsf{t})=0$ and $P(\mathsf{t})$ is empty (i.e.\ the root vertex and distinguished vertex coincide) 
or if $e(\mathsf{t})=n-k$ and $p(\mathsf{t})=1$.  Note that if $e(\mathsf{t})=n-k$ then $\nu(\mathsf{t})=1$, since this is the maximum possible number of edges (by Lemma $\ref{bounds}$).
	
	Reach defines a filtration because contracting an edge which is not adjacent to a vcd vertex lowers the reach by at least one, while contracting an edge which is adjacent to a vcd vertex lowers the reach by $1$ if said edge is not on the 0-to-DV path and it preserves the filtration degree if the edge is on the $0$-to-DV path.  Thus in all cases the reach is not increased by the differential and so we have an increasing filtration
	$$
	\op{A}_0\subset \op{A}_1\subset... \subset\op{A}_r\subset\op{A}_{r+1} \subset ... \subset \op{A}_{2n-2k-2}$$
	where $\op{A}_r$ denotes the span of those Stirling trees in $\op{A}$ of reach $\leq r$.
	
	Consider the associated graded.  Each $\op{A}_{r+1}/\op{A}_r$ may be described as the span of trees of reach equal to $r$, with differential given by projection of the original differential to the quotient.  By the above analysis, all differential terms vanish in this quotient, except in the case where we have a vcd vertex and we contract the unique edge of maximum height on the 0-to-DV path (i.e.\ the edge immediately below the distinguished vertex).  Let us call this differential $\delta_0$.

\begin{figure}
	\centering
	\includegraphics[scale=1.15]{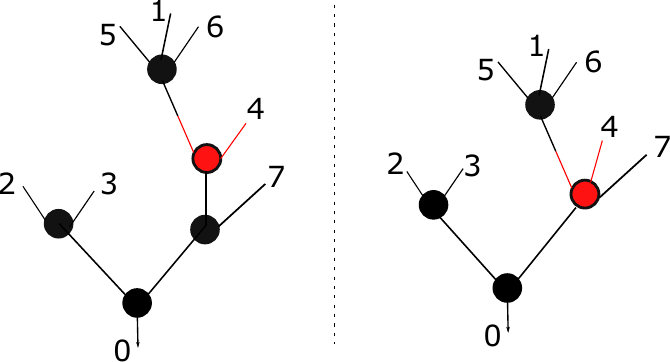}
	\caption{The Stirling tree of type $(7,2)$ on the left corresponds to a $1$-dimensional subspace of $\op{A}$ having reach $r = 8-2-1 = 5$.  In the associated graded $\op{A}_5/\op{A}_4$, its unique non-zero differential term lands in the summand corresponding the Stirling tree on the right, which we confirm has reach $r=6-1-0=5$.}
	\label{fig:sta}
\end{figure}
	
We therefore conclude that each Stirling tree of type $(n,k)$ satisfying the properties (i) and (ii) above corresponds to a $1$-dimensional subspace of the complex $\op{A}$ which is either the source or the target of exactly one non-trivial differential $\delta_0$ in the associated graded.  Indeed, each quotient complex $\op{A}_{r+1}/\op{A}_r$ splits as a direct sum of complexes of the form 
	\begin{equation}\label{edge2}
0\to det(E(\mathsf{t}))\tensor det(A_\mathsf{t}) \stackrel{\delta_0}\longrightarrow det(E(\mathsf{t}/e))\tensor det(A_\mathsf{t/e}) \to 0
	\end{equation}
where $\mathsf{t}$ is a Stirling tree of type $(n,k)$ and reach $r$ whose distinguished vertex is a vcd vertex and where $e$ is the unique edge on the $0$-to-DV path of maximum height.  See Figure $\ref{fig:sta}$.  Since the differential in Equation $\ref{edge2}$ is non-zero, it must be an isomorphism.  Thus the associated graded has no homology, and so neither does $\op{A}$.

Having shown that $\op{A}$ has no homology, to compute the homology of $\op{S}_{n,k}$, it suffices to compute the homology of the quotient $\op{S}_{n,k}/\op{A}$.  
The quotient complex is spanned by those Stirling trees which satisfied neither condition (i) nor condition (ii) above.  This means such a tree must have the root vertex equal to the distinguished vertex and be of vcd type. 
In the quotient, the non-zero differential terms can only contract edges that are not adjacent to the distinguished vertex.  To each such tree, associate a
 partition of the set $\{1\cdc n\}$ by declaring two numbers to belong to the same block if and only if the leaves they label share the same nearest alternating flag.  By definition, the nearest alternating flag is the input flag adjacent to the root vertex that we first encounter when we move along the unique shortest path connecting a leaf to the root.  In particular, the blocks in the partition must be non-empty, and the number of blocks is $k$.  

Such a partition is unchanged by the differential in the quotient $\op{S}_{n,k}/\op{A}$, and so this complex splits over the choice of such partitions.    Therefore to compute the homology of the quotient, it suffices to compute the homology of each summand.  For this, we appeal to Koszulity of the commutative operad.  Indeed each summand may be viewed as a tensor product over the blocks of the partition of a complex isomorphic to the bar construction of the commutative operad in arity equal to the size of the block.  This is simply because each branch emanating from an alternating flag has only standard edge contractions, with no distinguished vertices.  Appealing to Koszul duality of the commutative and Lie operads, denoted $\mathsf{Com}$ and $\mathsf{Lie}$,  the homology of a summand corresponding to a partition $b_1\cdc b_k$ is isomorphic to $\tensor_i \mathsf{Lie}(|b_i|)$, concentrated in the maximal degree, shown above to be $n$. 

Finally we recall that the dimension of the constituent arities of the Lie operad satisfy $\mathsf{Lie}(|b_i|) = (b_i-1)!$, which is also the number of cycles of length $b_i$ that can be formed from a set of size $b_i$.  Therefore, the dimension of the homology in degree $n$ is equal to the number of ways to partition $\{1\cdc n\}$ into $k$ non-empty blocks, and to subsequently associate a cycle of length $|b_i|$ with the numbers appearing in the block $b_i$.  In other words, $\beta_n(\op{S}_{n,k})$ is the number of permutations on $n$ letters having $k$ cycles, or $|s_{n,k}|$.
\end{proof}

\begin{remark}
	As mentioned above, Proposition $\ref{alt}$ can be seen as a corollary of this Theorem.  For this, compute the Euler characteristic $\chi(\Sigma^{-k}\op{S}_{n,k})$ in two ways.  As the alternating sum of Betti numbers we find $s_{n,k}$.  On the other hand, if we compute the alternating sum of the dimensions of the complex itself, we can use Koszulity of the commutative operad to collapse subtrees not containing a distinguished vertex, and label the vertices resulting from the collapse with the space of Lie words of appropriate arity.
	The degree of the result is determined by the valence of the distinguished vertex, call it $m$, and is equal to $n+1-m$.  The number of all such trees, after this collapse, is counted by the ways to partition the set $\{0\cdc n\}$ times the dimension of the space of Lie words on each block of the partition which, as the proof above indicates, is $|s_{n+1,m}|$, times the possible choices of alternating flags, which is ${m-1 \choose k}$.  We therefore see the alternating sum arising on the right hand side of Equation $\ref{tweedie}$.

\end{remark}

\section{Application to Commutative Graph Homology}\label{ghsec}  In this final section we would like to say how the Stirling complexes arose as a natural object of study, and what Theorem $\ref{main}$ tells us in this context.  Specifically, the complexes $\op{S}_{n,k}$ are isomorphic to subcomplexes of the Feynman transform of Lie graph homology in genus $1$.  We will give a recollection of these prerequisites, and we refer to \cite{WardMP} and \cite{WardWheels} for more details.

\subsection{Modular Operads}
Modular operads were introduced by Getzler and Kapranov in \cite{GeK2} as a generalization of the notion of an operad.  Modular operads permit composition along all graphs, not just along trees.  Here we consider graphs whose vertices carry a labeling by a non-negative integer, called the genus of a vertex, which allows us to keep track of loop contractions.  We call these modular graphs (see subsection $\ref{modgraphs}$).  Similar to an operad, compositions are generated by single edge contractions, hence the following definition.

\begin{definition}
	A (dg) modular operad $\mathsf{M}$ is a collection of $S_n$ modules $\mathsf{M}(g,n)$ indexed over pairs of natural numbers $g$ and $n$ with $2g+n\geq 3$, along with two families of linear maps called loop contractions and bridge contractions:
	\begin{itemize}
		\item Loop contractions: for all $g$, all $n\geq 2$ and all pairs $1\leq i<j\leq n$, there exists a linear map $$\circ_{ij}\colon\mathsf{M}(g,n) \to \mathsf{M}(g+1,n-2).$$
		\item Bridge contractions: for all $g_1,g_2,n_1,n_2$ and each $1\leq i \leq n_1$ and $1\leq j \leq n_2$ there exists a linear map $$\xycirc{i}{j}\colon \mathsf{M}(g_1,n_1)\tensor \mathsf{M}(g_2,n_2) \to \mathsf{M}(g_1+g_2,n_1+n_2-2).$$  
		\end{itemize}
\end{definition}
The terminology ``bridge'' and ``loop'' contractions comes from the classification of edges in a graph as either bridges or loops (Appendix $\ref{graphs}$).  Collectively we call these compositions ``edge contractions''.

Edge contractions are posited to satisfy a series of axioms analogous to the associativity and $S_n$-equivariance axioms for an operad.  These axioms should be intuitively understood via the graphical intuition which views each $\mathsf{M}(g,n)$ as a vertex with label $g$ and valence $n$.  The loop contraction $\circ_{ij}$ is viewed as gluing the legs labeled $i$ and $j$ to form a ``loop'' before contracting the result.  The bridge contraction $\xycirc{i}{j}$ glues together the corresponding legs on two vertices to form a ``bridge'' before contracting the result.  The associativity axioms assert that contracting edges of a graph in any order produces the same result, regardless of the order of the contractions.

With this intuition the associativity axioms for a modular operad are indexed by graphs with two edges and are divided into the following four families.  See also Figure $\ref{fig:assoc}$.

\begin{itemize}
	\item Contraction of two bridges $\xycirc{i^\prime}{j^\prime}\cdot\xycirc{l}{k} = \xycirc{l^\prime}{ k^\prime}\cdot\xycirc{i}{j}$
	\item Contraction of a loop and a bridge $\circ_{i^\prime j^\prime}\cdot\xycirc{l}{k} = \xycirc{l^\prime}{k^\prime}\cdot\circ_{i j}$
	\item Contraction of parallel bridges $\circ_{i^\prime j^\prime}\cdot\xycirc{l}{k} = \circ_{l^\prime k^\prime}\cdot\xycirc{i}{j}$
	\item Contraction of two loops $\circ_{i^\prime j^\prime}\cdot\circ_{lk} = \circ_{l^\prime k^\prime}\cdot\circ_{i j}$
\end{itemize}
  The notation $i^\prime, j^\prime, l^\prime, k^\prime$ means the label of the leg formerly labeled by $i,j,k,l$.  The precise label will depend on a chosen convention for relabeling legs after edge contraction.  See \cite[Page 33]{WardMP} for one such convention.

\begin{figure}
	\centering
	\includegraphics[scale=.6]{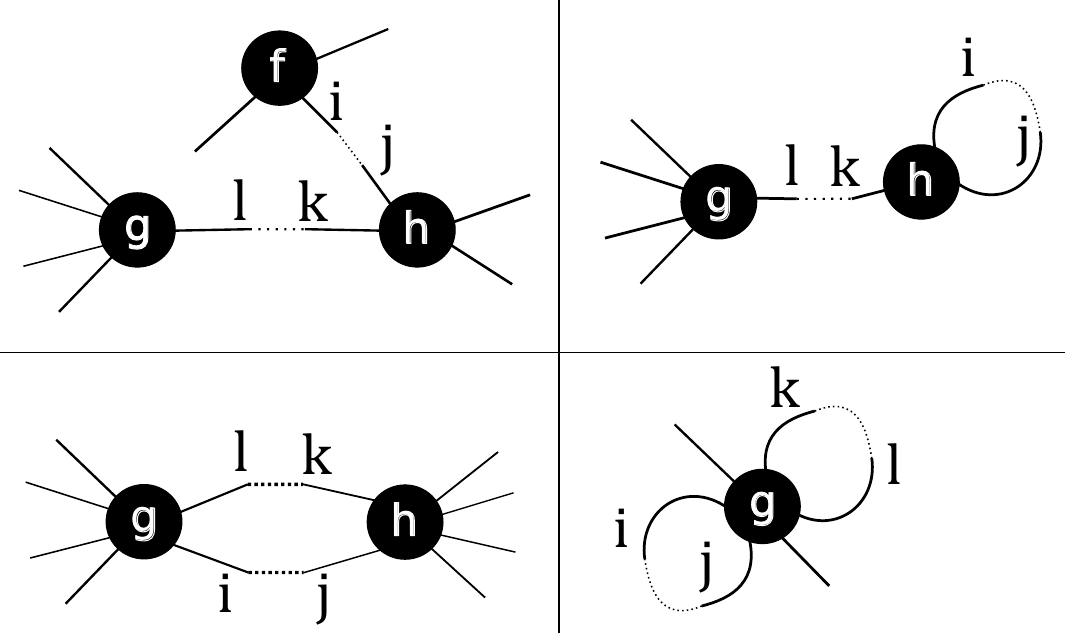}
	\caption{The associativity relations in a modular operad are indexed by graphs with two edges.  These relations are divided into four types, corresponding to the four pictured two-edged graphs.  The leg labels are suppressed except for those which were used in the formation of the two edges.  The formed edges are indicated with dashed lines.  The target of the composition of operations is determined by contracting the edges, while keeping track of the total genus via the genus label.}
	\label{fig:assoc}
\end{figure}

Finally, both loop contraction and bridge contraction satisfy compatibility with the symmetric group actions.  These compatibility conditions are of the form $
\circ_{ij}\cdot\sigma
= 
\circ_{ij}(\sigma) \circ_{\sigma^{-1}(i)\sigma^{-1}(j)} 
$ 
and $
\xycirc{i}{j}\cdot(\sigma,\tau)
= 
\sigma\xycirc{i}{j}\tau \cdot \xycirc{\sigma^{-1}(i)}{\tau^{-1}(j)} 
$ for certain permutations $\circ_{ij}(\sigma)$ and $\sigma\xycirc{i}{j}\tau$ determined by permutations $\sigma$ and $\tau$ and leg labels $i$ and $j$.  As in the case of an operad, these equivariance axioms may be understood by forming one-edged graphs as specified by the edge contraction and the permutation, and asserting that if the labeled one-edged graphs are the same then the operations are the same.

\subsection{The Feynman transform}

  The Feynman transform \cite{GeK2} is the modular operadic generalization of the bar construction.  Viewed as a functor, its input is a modular operad, call it $\mathsf{M}$, and its output may be understood as the direct sum over graphs whose vertices are labeled by $\mathsf{M}$,  (or by its linear dual depending on conventions).  The result is a family of chain complexes $\mathsf{FT}(\mathsf{M})(g,n)$ for each total genus $g$ and number of legs $n$, with differential given by a sum over edge contractions (or dually edge expansions).  When an edge is contracted, the modular operad structure of $\mathsf{M}$ is used to determine the new vertex label. 

To make the above description precise, we proceed as follows.  Let $\mathsf{M}$ be a modular operad.  Let $\gamma$ be a modular graph (see Appendix $\ref{modgraphs}$).  Let $V(\gamma)$ and $E(\gamma)$ denote the sets of vertices and (internal) edges of $\gamma$ respectively.  If $v\in V(\gamma)$ has genus $g(v)$, valence $n(v)$ and set of adjacent flags $A(v)$, we define a vector space $\mathsf{M}(v):=\mathsf{M}(g,A(v))$ where $\mathsf{M}(g,-)$ is viewed as a functor from the category of all finite sets and bijections via left Kan extension.  In particular $\mathsf{M}(v)\cong \mathsf{M}(g(v),n(v))$, but non-canonically so.  We then define
\begin{equation*}
\mathsf{M}(\gamma) := \ds\tensor_{v\in V(\gamma)} \mathsf{M}(v).
\end{equation*}

Additionally, to $\gamma$ we associate the one-dimensional graded vector space
$\mathfrak{K}(\gamma) := det(E(\gamma))$,
where $det$ was defined in subsection $\ref{detsec}$.  With these pre-requisites, we define the chain complex
\begin{equation}\label{FTeq}
\mathsf{FT}(\mathsf{M})(g,n) := \ds\bigoplus_{\substack{\text{iso.\ classes of } \\ (g,n)-\text{graphs } \gamma}} \mathsf{M}(\gamma)\tensor_{Aut(\gamma)} \mathfrak{K}(\gamma),
\end{equation}
with differential given by the sum over edge contractions.  Here $Aut(\gamma)$ denotes the group of leg fixing automorphisms of $\gamma$.      

Each chain complex $\mathsf{FT}(\mathsf{M})(g,n)$ is bigraded by the number of edges of the underlying graph and the sum of the degrees of the vertex labels, which we call the internal degree.  We write $\mathsf{FT}(\mathsf{M})(g,n)^{r,s}$ for the span of those graphs having $r$ edges and internal degree $s$.   The differential in the Feynman transform contracts one edge, using the degree $0$ modular operad structure maps, so it has the form
\begin{equation}\label{deq}
 \mathsf{FT}(\mathsf{M})(g,n)^{r,s}\stackrel{d}\longrightarrow  \mathsf{FT}(\mathsf{M})(g,n)^{r-1,s},
\end{equation}
and so the complex splits into a direct sum of subcomplexes indexed over internal degree.  The Stirling complexes, studied above, are examples of such subcomplexes for a particular choice of $\mathsf{M}$ as we shall subsequently explain.   

We remark that the Feynman transform as originally considered by Getzler and Kapranov \cite{GeK2} is the linear dual of the definition given here, so it may be more appropriate to call this the coFeynman transform.

\subsection{Lie Graph Homology}

Given a modular operad $\mathsf{M}$, the family of chain complexes $\mathsf{FT}(\mathsf{M})(g,n)$ can themselves be composed along graphs by freely grafting the legs of the graph together.  Thus the Feynman transform forms something akin to a modular operad, but the degree of the compositions is non-zero.  The notion of $\mathfrak{K}$-modular operads, in which composition formally has degree $1$, was also introduced in \cite{GeK2} as a target for the Feynman transform.  Conversely, an analogous construction produces a modular operad from a $\mathfrak{K}$-modular operad, and so we view the Feynman transform as a pair of functors
\begin{equation}
\left\{\text{modular operads} \right\} \stackrel{\mathsf{FT}}\leftrightarrows \left\{\mathfrak{K}\text{-modular operads} \right\}.
\end{equation}

As above, we write $\mathsf{Lie}$ for the Lie operad.  We may choose to view $\mathsf{Lie}$ as a modular operad by declaring the structure map corresponding to any graph of positive genus to be $0$.  Alternatively, we can take the shifted operadic suspension of the Lie operad and then extend the result by $0$ to higher genus, and the result will be a $\mathfrak{K}$-modular operad.  Denote this $\mathfrak{K}$-modular operad $\Sigma\mathfrak{s}^{-1}\mathsf{Lie}$.  By ``Lie graph homology'' in this article we refer to the modular operad $H_\ast(\mathsf{FT}(\Sigma\mathfrak{s}^{-1}\mathsf{Lie}))$.

Lie graph homology was studied in \cite{CHKV}, where it was shown to be calculable as the group homology of a family of groups $\Gamma_{g,n}$, and so we use notation 
\begin{equation*}
H_\ast(\Gamma)(g,n) :=H_\ast(\Gamma_{g,n}) = H_\ast(\mathsf{FT}(\Sigma\mathfrak{s}^{-1}\mathsf{Lie}))(g,n),
\end{equation*}
and the notation $H_\ast(\Gamma)$ for the associated modular operad.

Below, we will connect the Stirling complexes to the genus $1$ component of the Feynman transform of the modular operad $H_\ast(\Gamma)$.  This in turn depends only on the graded vector spaces $H_\ast(\Gamma_{g,n})$ for $g\in \{0,1\}$ and the edge contractions between them.  This structure can be completely characterized as follows. 

First the homology of these spaces is given by:
\begin{lemma}\label{hlem}\cite{CHKV}  When $g\leq 1$, the homology of $\Gamma_{g,n}$ is:
	$$H_k(\Gamma_{0,n}) = \begin{cases} V_n & \text{ if } k=0 \\ 0 & \text{ else } \end{cases}$$
 and
 	$$H_k(\Gamma_{1,n}) = \begin{cases} V_{n-k,1^k} & \text{ if } k \text{ is even} \text{ and } 0\leq k\leq n-1 \\ 0 & \text{else} \end{cases}$$
\end{lemma}

Second, to characterize the modular operadic structure maps between these spaces, we first observe that edge contractions between degree $0$ vertices will be isomorphisms.  Thus the only non-trivial structure maps landing in genus $\leq 1$ are contractions along bridge edges adjacent to a genus $0$ and genus $1$ vertex.  To characterize these operations, we describe the vector spaces $H_k(\Gamma_{1,k+l})$ in terms of such compositions as follows:

\begin{lemma}\cite{WardWheels}\label{basislem}  Fix any non-zero vectors $\mu_l\in H_0(\Gamma_{0,l+1})$ and $\alpha_k\in H_k(\Gamma_{1,k+1})$.  The vector space $H_k(\Gamma_{1,k+l})$ has a basis consisting of edge contractions of the form $\sigma(\mu_l\circ_1 \alpha_k)$ over all $\sigma \in S_{k+l-1}$.
\end{lemma}

To be specific, this result is a corollary of \cite[Lemma 3.2]{WardWheels} which gives the precise form of the relations as well.

Here we choose to view our graphs as labeled by $\{0\cdc n-1\} \cong \{1\cdc n\}$ via the isomorphism sending $0$ to $n$ and fixing those $i$ with $0<i<n$.  Then $\circ_1:= \xycirc{1}{0}$ denotes the operadic composition which glues the root (by convention the leaf labeled by $0$) of $\alpha$ on to the leaf labeled by $1$ of $\mu$.  The permutations $\sigma$ fix $0$, hence the lemma says that has a basis given by those compositions in-which $0$ is adjacent to the root (see Figure $\ref{fig:edge2}$).

\begin{figure}
	\centering
	\includegraphics[scale=.95]{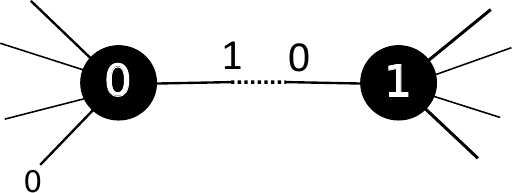}
	\caption{Compose the two indicated vertices via the composition $\circ_1$. 
	The genus $1$ vertex carries a label in $H_k(\Gamma_{1,k+1})$ and the genus $0$ vertex carries a label in $H_0(\Gamma_{0,l+1})$.  After choices of generators of these one dimensional vector spaces, call them $\alpha_k$ and $\mu_l$ respectively, each way to label the legs in the pictured tree results in a class in $H_k(\Gamma_{1,k+l})$.  Lemma $\ref{basislem}$ says that the ${k+l-1} \choose {k}$ such classes form a basis.  }
	\label{fig:edge2}
\end{figure}

\subsection{Stirling complexes and the Feynman transform of Lie graph homology.}

With an understanding of the modular operad $H_\ast(\Gamma)$ in genus $\leq 1$, we now turn to consideration of the chain complex $\mathsf{FT}(H_\ast(\Gamma))(1,n)$.  By Equation $\ref{deq}$, we know that this complex splits over internal degree.  By Lemma $\ref{hlem}$ we know that it is non-zero only when the internal degree is even.  The following theorem relates the Stirling complexes to the summands corresponding to non-zero internal degree.

\begin{theorem}\label{isothm}  Let $k$ be a positive, even integer.  There is an isomorphism of chain complexes
$$\mathsf{FT}(H_\ast(\Gamma))(1,n+1)^{\ast,k}\cong \op{S}_{n,k}.$$
\end{theorem}
\begin{proof}  Fix such a $k$, and assume $k\leq n$, else the statement is vacuous.  Consider the chain complex $\mathsf{FT}(H_\ast(\Gamma))(1,n+1)^{\ast,k}$.  By Equation $\ref{FTeq}$, this chain complex is a direct sum indexed over graphs of type $(1,n+1)$, i.e.\ total genus $1$ and with $n+1$ labeled legs.

By definition, the total genus of a modular graph is the sum of the genera of the vertex labels with the first Betti number of the given graph.  If the vertex labels all have genus $0$, then the internal degree would be $0$ (by Lemma $\ref{hlem}$).  Since this is not the case, there must be a vertex of positive genus.  Since the total genus is $1$, this can only happen if the graph is a tree and there is a unique vertex of genus $1$, with all other vertices of genus $0$.  We thus conclude that the homogeneous elements in $\mathsf{FT}(H_\ast(\Gamma))(1,n+1)^{i,k}$ correspond to trees with $i$ edges and $n+1$ labeled legs, along with a distinguished vertex (namely the unique vertex of genus $1$).

Given such an homogeneous element, the degree $0$ vertices of valence $l+1$ are labeled with homology classes from $H_\ast(\Gamma_{0,l+1})$ which is simply the ground field concentrated in degree $0$.  The genus $1$, or distinguished, vertex is therefore labeled with a homology class in $H_k(\Gamma_{1,m})$ for some $m$.

Using the basis of Lemma $\ref{basislem}$, we may thus view an homogeneous element in $\mathsf{FT}(H_\ast(\Gamma))(1,n+1)^{i,k}$ as a tree with a distinguished vertex which is itself labeled by a tree with one edge and two vertices.  Among these two ``inner'' vertices, the upper one is labeled by the genus $1$ alternating class $\alpha_k$ and the lower one is labeled by the genus $0$ class $\mu_l$.  In order to specify a Stirling tree from this data, we color all legs adjacent to the genus $1$ vertex red, and all the others black.  See Figure $\ref{fig:stc}$.

\begin{figure}
	\centering
	\includegraphics[scale=1.15]{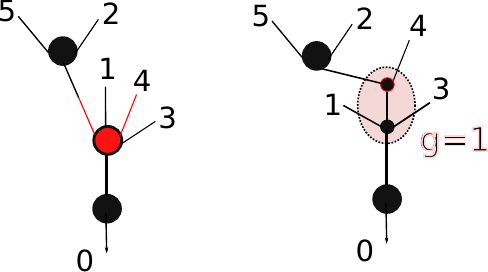}
	\caption{The Stirling tree on the left corresponds to a modular graph of total genus $1$ on the right.  The unique genus $1$ vertex (with dotted boundary) is itself labeled by a composition of the commutative product $\mu_l$ with the alternating class $\alpha_{k}$.  (Pictured is the case $k=2$.)  This composition specifies the label of the genus $1$ vertex.}
	\label{fig:stc}
\end{figure}

This gives an isomorphism of graded vector spaces.  In hindsight, the formula for the differential of the Stirling complexes given above (Subsection $\ref{diffsec}$) was chosen such that this isomorphism of graded vector spaces is also an isomorphism of chain complexes.  To see that this is the case, observe that both differentials are given by sums over edge contraction, and it's immediate that the terms coincide when contracting an edge which is not adjacent to the distinguished vertex, or when contracting an edge which is adjacent to the distinguished vertex but is non-alternating/adjacent to the lower inner vertex.  The only non-trivial case, then, is to show that contracting an edge adjacent to the upper/alternating inner vertex corresponds to the description given above for contracting an edge in a Stirling tree when said edge contains an alternating flag.  This in turn requires rewriting a one-edge composition landing in $H_k(\Gamma_{1,m})$ in terms of the chosen basis.  A formula for this rewriting rule is given in \cite[Lemma 3.2]{WardWheels}, and the description in terms of Stirling trees is simply a reformulation of the statement of that result.  \end{proof}

Combining this result with Theorem $\ref{main}$, we have the following corollary pertaining to commutative graph homology.  Let $\overline{\mathsf{Com}}$ be the modular operad defined by $\overline{\mathsf{Com}}(g,n) = \mathbb{Q}$, all of whose structure maps are the canonical isomorphisms.  

\begin{corollary}\label{gccor} There is an isomorphism of $S_{n+1}$-modules $$H_{\ast+1}(\mathsf{FT}(\overline{\mathsf{Com}}))(1,n+1)\cong \bigoplus_{i=1}^{\floor{n/2}} H_\ast(\op{S}_{n,2i}).$$
\end{corollary}

\begin{proof}  
Consider the chain complex 
$$\FTGK(H_\ast(\Gamma))(1,n+1)
= \FTGK(H_\ast(\FTGK(\Sigma \mathfrak{s}^{-1} \mathsf{Lie})))(1,n+1).
$$  If it were the case that the modular operad $\FTGK(\Sigma \mathfrak{s}^{-1} \mathsf{Lie})$ were formal, then this complex would be acyclic, because $\FTGK^2(\Sigma\mathfrak{s}^{-1} \mathsf{Lie}(1,n+1))$ has the same homology as $\Sigma\mathfrak{s}^{-1} \mathsf{Lie}(1,n+1)$, which is zero by definition.  Alas, this modular operad is not formal (indeed we've computed a portion of its homology above).  However, in this setting, the bicomplex $\FTGK(H_\ast(\Gamma))(1,n+1)^{r,k}$
is the $E_0$ page of a spectral sequence converging to $0$ (called the genus label filtration spectral sequence, \cite[Section 4.5.1]{WardMP}).  

The higher pages of this spectral sequence have differentials inherited from the homotopy transfer theorem applied to the modular operad $H_\ast(\Gamma)$.  These differentials necessarily increase the internal degree.  However, above the row $k=0$, the $E_1$-page is concentrated in total degree $n$ (combining Theorems $\ref{main}$ and $\ref{isothm})$.  Therefore, the $k=0$ row of the $E_1$ page must be concentrated in degree $n+1$, and moreover must be isomorphic to the direct sum of the homologies of $\op{S}_{n,k}$ over all even $k>0$.  In particular, from Theorem $\ref{main}$ we conclude that the homology of $\FTGK(H_\ast(\Gamma))(1,n+1)^{\ast,0}$ is concentrated in a single degree and has rank equal to the number of even (or odd) permutations of $S_{n}$, namely $n!/2$.	

Finally we observe that the modular operads $\overline{\mathsf{Com}}$ and $H_0(\Gamma)$ (taking only the degree 0 summands in each biarity) coincide, hence $$
\FTGK(H_\ast(\Gamma))(1,n+1)^{\ast,0}= \FTGK(H_0(\Gamma))(1,n+1) = \FTGK(\overline{\mathsf{Com}})(1,n+1),
$$
which finishes the proof.	
\end{proof}

\begin{remark}  Before proving Theorem $\ref{main}$ above, I had calculated the result by hand for low values of $n$.  For example the case of $n+1=5$ is pictured in Figure 14 of \cite{WardMP}, but Theorem $\ref{main}$ was needed to ensure that the analogous picture holds for all $n$.
\end{remark}

The chain complexes which comprise the $\mathfrak{K}$-modular operad $\mathsf{FT}(\overline{\mathsf{Com}})$ coincide (up to a shift in degree) with the reduced cellular chains of the spaces $\Delta_{g,n}$, and so this corollary establishes Equation $\ref{decomp}$ given in the introduction.

\subsection{Representation Stability}\label{rs}

As mentioned in the introduction, Corollary $\ref{gccor}$ agrees with \cite{CGP2} on rank, but gives a different perspective on the $S_{n+1}$-module structure.  We end this article with a modest conjecture:

\begin{conjecture}\label{rep} For each $i$ the sequence of $S_{n+1}$-modules
$H_n(\op{S}_{n,n-i})$ is representation stable, after tensoring with the alternating representation.
\end{conjecture}

Although filling in all the details is a bit beyond the scope of this article, let me give some idea of why this seems likely.  The first step would be to show that for each $r$, the sequence $\op{S}_{n,n-i,r}\tensor V_{1^n}$ is a finitely generated FI-module.  Here $r$ is the number of edges.  Recall \cite{CEF} that an FI-module is simply a functor from the category of finite sets and injections.  For this, given an injection $\iota\colon\{0 \cdc n\}\hookrightarrow \{0 \cdc n+l\}$, and a tree whose leaves are labeled $\{0\cdc n\}$, we can simply relabel leg $j$ with $\iota(j)$, and then add alternating flags to the distinguished vertex which are labeled by the elements not in the image of $\iota$.  The factor of $V_{1^n}$ ensures that the order of the added flags plays no role.

This putative FI-module would then be finitely generated by those trees whose distinguished vertex has valence $k+1$, and so the general theory of finitely generated FI-modules \cite{CEF} would apply to show that each sequence $\op{S}_{n,n-i,r}$ is representation stable.  Once this is established, it would be straight-forward to argue that the $S_{n+1}$-equivariant Euler characteristic, and hence the homology, is representation stable as well.  Having established representation stability, a next step would then be to determine the stable range and calculate some stable multiplicities.

Let us give some confirmation of this conjecture in the case of small $i$.  When $i=0$, we have already seen that
$H_n(\op{S}_{n,n})$ is the alternating representation of $S_{n+1}$, which is representation stable after tensoring with the sign representation.  

When $i=1$, each complex $\op{S}_{n,n-1}$ is only non-zero in two degrees corresponding to number of edges $r=0$ and $r=1$.  From \cite{CHKV} we know
$\op{S}_{n,n-1,0} \cong V_{2,1^{n-1}}$ 
and we may compute  
$\op{S}_{n,n-1,1}$ 
as an induced representation using the Littlewood-Richardson rule \cite{FH} to be $V_{2,1^{n-1}}\oplus V_{3,1^{n-2}}$.  Thus, by Theorem $\ref{main}$ the homology is  $H_n(\op{S}_{n,n-1})\cong V_{3,1^{n-2}}$, which again is representation stable after tensoring with the alternating (aka sign) representation.

Finally, when $i=2$ we can determine $H_n(\op{S}_{n,n-2})$ via Theorem $\ref{main}$ by calculating the $S_{n+1}$-equivariant Euler characteristic.  This requires $n\geq 4$ and just by dimension considerations, after Lemma \ref{basics} (3), $n=4$ must be outside the stable range.  However, in the case $n=5$ we find, by a slightly tedious tabulation of characters, that
$$H_n(\op{S}_{n,n-2})\cong V_{3,3}\oplus V_{2,2,1^2}\oplus V_{3,2,1}\oplus V_{5,1},$$
and using the hook length formula it is a straight forward computation to show that the stable sequence of $S_n$-modules which continues this $S_6$-representation (after tensoring with the sign representation) has dimension equal to $\ds\frac{1}{4}\ds{n \ds\choose 3}(3n-1)$, as predicted by Lemma $\ref{basics}$.

To conclude, let me remark that it would be interesting to use the semi-classical modular operadic character \cite{GetSC} to compute the character polynomials of each of these putative representation stable sequences.  Indeed, if we restrict attention to the modular operad whose only non-zero terms are $H_0(\Gamma_{0,n})$ and $H_k(\Gamma_{1,n})$, ranging over all $n$, the character of the Feynman transform can be read off from the results of {\it op.\ cit.}, and Conjecture $\ref{rep}$ implies, after \cite[Theorem 1.5]{CEF}, that the characters of these sequences will eventually be polynomials.


\appendix

\section{Trees and Graphs}
In an attempt to keep this article reasonably self-contained, we spell out our conventions regarding graphs and trees in this appendix.  We refer to e.g.\ \cite{BM}, \cite{KW} and \cite{WardMP} for more detail about this viewpoint.

\subsection{Graphs: Definition and basic terminology}\label{graphs}

\begin{definition}  A graph $\gamma=(V,F,a,\iota)$ is the following data:
	\begin{itemize}  
		\item A finite non-empty set $V$ called the vertices of the graph.
		\item A finite set $F$ called the flags of the graph.
		\item A function $a\colon F \to V$.
		\item A function $\iota\colon F \to F$ with the property that $\iota^2$ is the identity function.\end{itemize}
\end{definition}

Fix a graph $\gamma=(V,F,a,\iota)$ and let $v\in V$ and $f\in F$.  When discussing graphs we use the following terminology.  If $a(f)=v$ we say that $f$ is {\bf adjacent} to $v$.  The {\bf valence} of $v$ is $|a^{-1}(v)|$.  If $\iota(f)=f$ we say $f$ is a {\bf leg} of $\gamma$.  If $\iota(f)\neq f$ we say that the unordered pair $\{f,\iota(f)\}$ is an {\bf edge} of $\gamma$.  The set of edges of $\gamma$ is denoted $E(\gamma)$ or just $E$ if the context is clear.  The set of legs of $\gamma$ will be denoted $\text{leg}(\gamma)$.  Observe 
$
|F|=2|E|+|\text{leg}(\gamma)|.
$
An edge $\{f,\iota(f)\}$ will be called a {\bf loop} if $a(f)=a(\iota(f))$ and will be called a {\bf bridge} if $a(f)\neq a(\iota(f))$.  


A {\bf path} in a graph $\gamma$ is a finite sequence of flags $\{f_1\cdc f_r\}$ such that for each pair $f_i$ and $f_{i+1}$ either $a(f_i)=a(f_{i+1})$ or $\iota(f_i)=f_{i+1}$.  The length of a path is the number of flags in the sequence.  Given such a path, if $a(f_1)=v$ and $a(f_r)=w$ then the path is said to run from $v$ to $w$.  We specifically allow the empty path from any vertex to itself, considered to be of length $0$.

A {\bf subgraph} of $\gamma$ is a pair of subsets $V^\prime \subset V$ and $F^\prime \subset F$ which are closed under $a$ and $\iota$.    In a graph $\gamma$, define an equivalence relation on the vertices $V$ via $v\sim w \iff \exists $ a path running from $v$ to $w$.  
Let $V_1\cdc V_c$ be the equivalence classes of this relation.  Partition $F$ by defining $F_i:=\{f\in F: a(f)\in V_i\}$.  It is straightforward to show that each $(V_i,F_i)$ is a subgraph of $\gamma$, we call these subgraphs the {\bf connected components} of $\gamma$.  The {\bf genus} of a graph having $c$ components is defined to be $\beta_1(\gamma)=c-|V|+|E|$.  If $c=1$ we say that $\gamma$ is {\bf connected.}  It is an elementary exercise to show that given a connected graph of genus $0$ along with vertices $v$ and $w$, there is a unique shortest path running from $v$ to $w$.

\subsection{Modular Graphs}\label{modgraphs}

A {\bf genus labeled} graph is a graph along with a function $g\colon V\to \mathbb{Z}_{\geq 0}$.  We call the non-negative integer $g(v)$ the genus label of the vertex.  If $\gamma$ is a genus labeled graph, we define its {\bf total genus} to be  
$$g(\gamma) := \beta_1(\gamma)+ \ds\sum_{v\in V} g(v).$$  A genus labeled graph is said to be {\bf stable} if $2g(v)+|a^{-1}(v)|\geq 3$ for each $v\in V.$

A {\bf modular graph} of type $(g,n)$ refers to a stable, connected graph of total genus $g$ along with a bijection $\text{leg}(\gamma)\cong \{1\cdc n\}$.  We call the number associated to a given leg its labeling.  Morphisms, in particular isomorphisms, of modular graphs are presumed to preserve the genus labels and leg labels.

\subsection{Trees}\label{trees}

A {\bf tree} is a connected graph of genus $0$.  Trees will often be denoted $\mathsf{t}$.  A tree is stable if $|a^{-1}(v)|\geq 3$ for each $v\in V$.   An $n$-tree is a tree with $n+1$ legs, along with a bijection between the sets $\{0\cdc n\}$ and $\text{leg}(\mathsf{t})$, called the leg labeling.  The leg labeled by $0$ is called the root of an $n$-tree.  The vertex adjacent to the root is called the root vertex.  

In an $n$-tree, each vertex has a distinguished adjacent flag called the output.  The output of the root vertex is defined to be the root.  For any other vertex, call it $v$, the output is defined to be the first flag on the unique shortest path running from $v$ to the root vertex.  All flags which are not output flags are called input flags.  Morphisms, in particular isomorphisms, of $n$-trees are presumed to preserve the leg labels.  In particular, relabeling the legs can produce a non-isomorphic $n$-tree.

By convention, a stable $n$-tree determines a modular graph of type $(0,n+1)$ by declaring that each vertex is of genus $0$ and by relabeling the legs according to the isomorphism $\{0\cdc n\} \cong \{1\cdc n+1\}$ which sends $0$ to $n+1$ and fixes the set $\{1\cdc n\}$.  

\subsection{Graphical depiction of a graph.}

To a graph we may associate a diagram in which the vertices are depicted as nodes and the flags are depicted as arcs or line segments.  If $a(f)=v$, the line segment corresponding to $f$ is drawn adjacent to $v$ at one end, call it the vertex end.  If $f\neq \iota(f)$ then these two flags are joined at their non-vertex ends to form a graphical depiction of the edge.  The legs of the graph are depicted as line segments with loose ends.  Leg and genus labels may be added to the diagram to depict modular graphs and/or $n$-trees (see Figure $\ref{fig:diagram}$).  A graph can in turn be extracted from such a diagram, although the underlying graph depends only on the combinatorics of the diagram and not on any embedding in Euclidean space. 

\begin{figure}
	\centering
	\includegraphics[scale=.95]{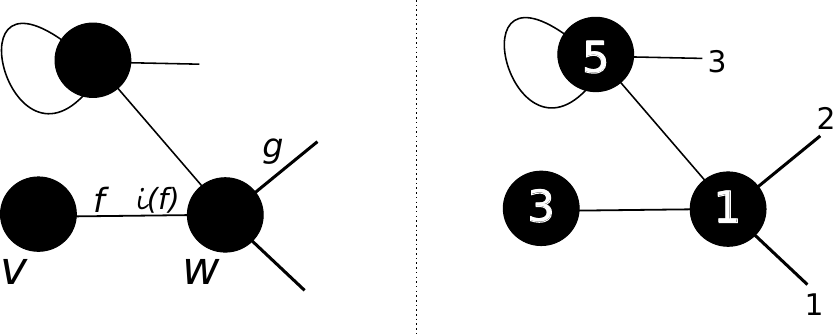}
	\caption{Left: a diagramatic depiction of a graph with three vertices and 9 flags having two bridges, one loop and three legs.  The labeled vertices $v$ and $w$ and labeled flags $f$ and $g$ satisfy $a(f)=v$, $a(g)=w$ and $\iota(g)=g$.
	Right:  adding labels to indicate a modular graph of type (10,3). }
	\label{fig:diagram}
\end{figure}

\end{document}